\documentclass[a4paper, 12pt]{amsart}
\usepackage[letterpaper, margin=.99in]{geometry}
\usepackage{latexsym}
\usepackage{amssymb}
\usepackage{amsmath}
\usepackage{amsfonts}
\usepackage{amsthm}

            \DeclareFontFamily{U}{wncy}{}
            \DeclareFontShape{U}{wncy}{m}{n}{%
               <5>wncyr5%
               <6>wncyr6%
               <7>wncyr7%
               <8>wncyr8%
               <9>wncyr9%
               <10>wncyr10%
               <11>wncyr10%
               <12>wncyr6%
               <14>wncyr7%
               <17>wncyr8%
               <20>wncyr10%
               <25>wncyr10}{}

\newtheorem{thm}{Theorem}[section]

\newtheorem{lem}[thm]{Lemma}

\newtheorem{cor}[thm]{Corollary}
\newtheorem{prop}[thm]{Proposition}
\theoremstyle{definition}
\newtheorem*{dfn}{Definition}
\newtheorem{remark}{Remark}

\newtheorem{example}{Example}

\newcommand{\N}{\mathbb N}
\newcommand{\Z}{\mathbb Z}
\newcommand{\Q}{{\mathbb Q}}
\newcommand{\C}{\mathbb C}

\def\al{\alpha}

\begin{document}

\raggedbottom

\title[$\mbox{PSL}(2,\Z)$ Acting on Imaginary Fields]{On the number of orbits arising from the action of $\mbox{PSL}(2,\Z)$ on imaginary quadratic number fields}

\author{Muhammad Aslam}
\address[M. Aslam]{Department of Mathematics, King Khalid University,
P.O. Box 9004, Abha, Saudi Arabia} \email{draslamqau@yahoo.com}

\author{Abdulaziz Deajim}
\address[A. Deajim]{Department of Mathematics, King Khalid University,
P.O. Box 9004, Abha, Saudi Arabia} \email{deajim@kku.edu.sa, deajim@gmail.com}

\date{\today}

\begin {abstract}
For square-free positive integers $n$, we study the action of the modular group $\mbox{PSL}(2,\Z)$ on the subsets $\{\,\frac{a+\sqrt{-n}}{c}\in \Q(\sqrt{-n})\, | \, a,b=\frac{a^2+n}{c},c \in \Z \,\}$ of the imaginary quadratic number fields $\Q(\sqrt{-n})$. In particular, we compute the number of orbits under this action for all such $n$ as provide an interesting congruence property of this number. An illustrative example and a C$^{++}$ code to calculate such a number for all $1\leq n \leq 100$ are also given.
\end {abstract}
\maketitle

\section{{\bf Introduction}}\label{intro}

Throughout this paper, we denote by $G$ the modular group $\mbox{PSL}(2, \Z)$, whose elements are all the M\"{o}bius transformations $z\mapsto (az+b)/(cz+d)$, $a,b,c,d\in\Z$, $ad-bc=1$. It is known (\cite{HM}) that $G$ has the finite presentation $<x,y : x^2=y^3=1>$, where $x$ and $y$ are, respectively, the transformations $z\mapsto -1/z$ and $z \mapsto (z-1)/z$. The modular group belongs to a more general family of groups called Hecke groups. A Hecke group $H_n$, $3\leq n \in \N$, is the group generated by the two M\"{o}bius transformations $z\mapsto -1/z$ and $z\mapsto z+\lambda_n$, where $\lambda_n=2\cos(\pi/n)$. It can be shown that $G \cong H_3$. Actions of the modular group, and Hecke groups in general, on many discrete and non-discrete structures have played significant roles in different branches of mathematics (see \cite{CD} for example).

Among the important discrete structures upon which the modular group acts are quadratic number fields. For a {\it real} quadratic number field $L=\Q(\sqrt{m})$, Q. Mushtaq (in \cite{Mush1}) studied the action of $G$ on the following subset of $L$:
$$\{\,\frac{a+\sqrt{m}}{c}\in L \; | \; a,\frac{a^2-m}{c},c \in \Z \,\}.$$
Subsequent works by several authors considered properties emerging from this action (see for instance \cite{MR}, \cite{MZ}, and \cite{Mush2}).

We shift the emphasis in this work towards studying the action of the modular group on {\it imaginary} quadratic number fields. Throughout this paper, $n$ denotes a square-free positive integer. Let $K_{-n}$ be the imaginary quadratic number field $\Q(\sqrt{-n})$, and consider the following subset of $K_{-n}$:
$$M_{-n}:=\{\,\frac{a+\sqrt{-n}}{c}\in K_{-n}\; | \; a,b=\frac{a^2+n}{c},c \in \Z \,\}. $$
It can be checked that $M_{-n}$ is the collection of the complex roots of all quadratic polynomials of the form $cx^2-2ax+b$ of the fixed discriminant $-4n$, with $a,b,c\in \Z$ and $0\leq a^2 < bc$.



It is not hard to see that there is a natural action of $G$ on $K_{-n}$ (inherited from the action of $G$ on $\mathbb{C}$). M. Ashiq and Q. Mushtaq in \cite{Ash-Mush} studied the action of a certain {\it subgroup} of $G$ on $M_{-n}$.
The aim of this paper is to study the action of $G$ itself on $M_{-n}$ and, in particular, count the number of orbits in $M_{-n}$ emerging from this action and present an interesting congruence property of this number (Theorems \ref{orbits 1} and \ref{orbits 2}).

\section{{\bf The action of $G$ on $M_{-n}$}}\label{properties}

For $\alpha=\cfrac{a+\sqrt{-n}}{c} \in M_{-n}$, we use the notation $a_\alpha:=a$, $b_\alpha:=b$, and $c_\alpha:=c$, and we call the ordered triple $(a_\al, b_\al, c_\al)$ {\it the signature of $\al$}.

\begin{prop}
$M_{-n}$ is a $G$-set.
\end{prop}

\begin{proof}
As $G$ acts on $K_{-n}$, it remains only to show that $M_{-n}$ is invariant under this action. Let $\alpha=\cfrac{a+\sqrt{-n}}{c} \in M_{-n}$. To show that $t(\alpha) \in M_{-n}$ for every $t \in G$, it suffices to show that $x(\alpha), y(\alpha) \in M_{-n}$ since $\{x, y\}$ is a complete set of generators of $G$. We see, first, that
$$x(\alpha)=-1/\alpha = \frac{-c}{a+\sqrt{-n}}=\frac{-c(a-\sqrt{-n})}{a^2+n}=\frac{-a+\sqrt{-n}}{b}.$$ Noticing that $a_{x(\alpha)}=-a \in \Z$, $c_{x(\alpha)}=b \in \Z$, and $b_{x(\alpha)}=\cfrac{a_{x(\alpha)}^2+n}{c_{x(\alpha)}}=\cfrac{a^2 +n}{b}=c\in\Z$, we get that $x(\alpha)\in M_{-n}$.
Similarly, we see that
$$y(\alpha)=1-\frac{1}{\al}=1+x(\al)=1+\frac{-a+\sqrt{-n}}{b}=\frac{(-a+b)+\sqrt{-n}}{b}.$$
As $a_{y(\alpha)}=-a+b \in \Z$, $c_{y(\alpha)}=b \in \Z$, and $$b_{y(\alpha)}=\frac{a_{y(\alpha)}^2+n}{c_{y(\alpha)}}=\frac{(-a+b)^2+n}{b}=\frac{a^2-2ab+b^2+n}{b}=-2a+b+\frac{a^2+n}{b}=-2a+b+c \in\Z,$$
we get that $y(\alpha)\in M_{-n}$ as well.
\end{proof}

\begin{remark}\label{remark 1}
For some use in the sequel, the following table summarizes the action of each $t\in \{x, y, y^2\}$ on an arbitrary element $\alpha=\cfrac{a+\sqrt{-n}}{c} \in M_{-n}$. The first two lines of the table were verified in the above proof, while the third line can be checked in a similar manner.\\

\begin{center}
\begin{tabular}{c| c c c c c}
\hline
 $t(\al)$ & $a_{t(\al)}$ & \quad & $b_{t(\al)}$ & \quad & $c_{t(\al)}$\\
\hline
$x(\al)$ & $-a$ & $\quad$ & $c$ & $\quad$ & $b$\\
$y(\al)$ & $b-a$ & $\quad$ & $-2a+b+c$ & $\quad$ & $b$\\
$y^2(\al)$ & $c-a$ & $\quad$ & $c$ & $\quad$ & $-2a+b+c$\\
\hline

\end{tabular}
\end{center}
\begin{center} \tablename{$\;$1}: Signatures of $x(\al), y(\al),$ and $y^2(\al)$ \end{center}
\end{remark}

We recall and introduce here some needed terminology.

\begin{dfn}(see \cite{Aslam})\hfill
\begin{enumerate}
\item[1.] An element $\alpha \in M_{-n}$ is said to be {\it totally positive} (resp. {\it totally negative}) if $a_\alpha c_\alpha >0$ (resp. $a_\alpha c_\alpha <0$).
\item[2.] Define the map $\| . \|:M_{-n} \to \N\cup\{0\}$ by $\|\al\|=|a_\al|$. We call $\|\al\|$ {\it the norm of $\al$} (not to be confused with the classical notion of norm).
\end{enumerate}
\end{dfn}

\begin{dfn} For $\al\in M_{-n}$, we call the set $\{\al, y(\al), y^2(\al)\}$ a {\it totally positive triple in $M_{-n}$} if $\al, y(\al),$ and $y^2(\al)$ are all totally positive. Denote the set of totally positive triples in $M_{-n}$ by $T^+(-n)$.
\end{dfn}

\begin{remark}\label{remark 2} For $\cfrac{a+\sqrt{-n}}{c} \in M_{-n}$ and $b=\cfrac{a^2+n}{c}$, $bc$ is obviously always positive. Thus, $b$ and $c$ always have the same sign. So, an equivalent useful definition to the one given above can go like this: $\alpha \in M_{-n}$ is said to be totally positive if either $a_\alpha , b_\alpha, c_\alpha >0$ or $a_\al, b_\al, c_\al <0$; and $\alpha$ is said to be totally negative if either ($a_\al<0$ and $b_\al, c_\al >0$) or ($a_\al >0$ and $b_\al, c_\al <0$). Note that any $\al \in M_{-n}$ is either totally positive, totally negative, or has norm zero.
\end{remark}

\begin{example}
For $n=5$, $\al=(1+\sqrt{-5})/2\in M_{-5}$ is obviously totally positive. From Table 1, we have $y(\al)=(2+\sqrt{-5})/3$ and $y^2(\al)=(1+\sqrt{-5})/3$. It is clear that $y(\al)$ and $y^2(\al)$ are both totally positive as well. So, $\{\al, y(\al), y^2(\al)\}\in T^+(-5)$. Note, similarly, that $\al'=(-1+\sqrt{-5})/(-2)$, $y(\al')=(-2+\sqrt{-5})/(-3)$, and $y^2(\al')=(-1+\sqrt{-5})/(-3)$ are all totally positive and, thus, $\{\al', y(\al'), y^2(\al')\}\in T^+(-5)$.
\end{example}

As a matter of notation, for $\al\in M_{-n}$, we denote the orbit containing $\al$ in $M_{-n}$ under the action of $G$ by $\al^G$; that is $\al^G=\{\beta\in M_{-n}\;|\; \beta=t(\al) \; \mbox{for some $t\in G$}\}$. As the action of $G$ on every orbit is transitive, any element of the orbit can equally represent the orbit. This justifies the notation $\al^G$ for an orbit in $M_{-n}$ under the action of $G$. Denote the set of orbits in $M_{-n}$ under the action of $G$ by $\mathcal{O}^G(M_{-n})$; so $\mathcal{O}^G(M_{-n}):=\{\al^G\;|\; \al\in M_{-n}\}$. We adopt the standard notation $d(n)$ for the number of positive divisors of $n$.

Now we state below our two main results, which give formulas that count the number of orbits $\mathcal{O}^G(M_{-n})$ as well as an interesting congruence property of such a number.

\begin{thm} \label{orbits 1}
Let $n$ be a square-free positive integer. Then the number of orbits in $M_{-n}$ under the action of $G$ is:
\begin{align*}
|\mathcal{O}^G(M_{-n})|&= \left \{ \begin{array} {c@{\quad;\quad}l} 2 & \mbox{if $n=1$} \\
d(n)+|T^+(-n)| & \mbox{otherwise}.
\end{array} \right. \\
 &= \left \{ \begin{array} {c@{\quad;\quad}l} 2 & \mbox{if $n=1$}\\
 4 & \mbox{if $n=3$} \\
d(n)+(2/3)|A^+(-n)| & \mbox{otherwise,}
\end{array}\right.
\end{align*}
where $A^+(-n)$ denotes the set of signatures $$\{(a,b,c)\in \N^3\;|\; \cfrac{a+\sqrt{-n}}{c}\in M_{-n},\; b=\cfrac{a^2+n}{c}, \; b>a, \; c>a\}.$$ Moreover, $|\mathcal{O}^G(M_{-n})| \equiv 0\;(\mbox{mod}\,4)$ for $n\neq 1$ or $2$.
\end{thm}

For two positive integers $k\leq m$, denote by $d_{\leq k} (m)$ the number of positive divisors of $m$ which are less than or equal to $k$. For instance, $d_{\leq 4}(10)=2$ and $d_{\leq 10}(10)=d(10)=4$.

\begin{thm}\label{orbits 2}
Let $n > 3$ be a square-free integer. Then the number of orbits in $M_{-n}$ under the action of $G$ is:
$$|\mathcal{O}^G(M_{-n})|=d(n)+ \cfrac{2}{3}\;\sum_{i=1}^{\lfloor n/2 \rfloor}[ d(i^2+n)-2d_{\leq i} (i^2+n)].$$
\end{thm}

\section{{\bf Lemmas and Proofs of Theorems \ref{orbits 1} and \ref{orbits 2}}}

\subsection{{\bf Lemmas}}\hfill
Preparing for the proof of Theorems \ref{orbits 1} and \ref{orbits 2}, we consider some lemmas, some of which are interesting in their own right.

The following lemma shows that the sign of the denominators of elements in any given orbit is the same.
\begin{lem}\label{sign}
For $\al \in M_{-n}$, $\mbox{sign}(c_\beta)=\mbox{sign}(c_\al)$ for any $\beta \in \al^G$.
\end{lem}

\begin{proof}
It is sufficient to show that $c_{x(\al)}$, $c_{y(\al)}$, and $c_{y^2(\al)}$ have the same sign as $c_\al$. By Remark \ref{remark 2}, $b_\al$ and $c_\al$ have the same sign. Since $c_{x(\al)}=c_{y(\al)}=b_\al$ (Table 1), $c_{x(\al)}$ and $c_{y(\al)}$ have the same sign as $c_\al$. Since $b_{y^2(\al)}=c_\al$, $c_{y^2(\al)}$ have the same sign as $c_\al$ as well (because $c_{y^2(\al)}$ and $b_{y^2(\al)}$ have the same sign).
\end{proof}

The effect of the action of $x$ on elements of $M_{-n}$ and their norms is given below.

\begin{lem}\label{x}
Let $\al=\cfrac{a+\sqrt{-n}}{c}\in M_{-n}$. Then,
\begin{enumerate}
\item[1.] $\al$ is totally negative if and only if $x(\al)$ is totally positive.
\item[2.] $\|\al\|=\|x(\al)\|$.
\item[3.] $\al$ has norm zero if and only if $x(\al)$ has norm zero.
\end{enumerate}
\end{lem}

\begin{proof}\hfill
\begin{itemize}
\item[1.] From Table 1, notice that $a_{x(\al)}=-a$, $b_{x(\al)}=c$, and $c_{x(\al)}=b$. Suppose that $\al$ is totally negative. If $a<0$ and $b,c>0$, then it is clear that $a_{x(\al)}>0$ and $b_{x(\al)}, c_{x(\al)} >0$, which implies that $x(\al)$ is totally positive. The case $a>0$ and $b,c<0$ is similar. For the converse, suppose that $x(\al)$ is totally positive. If $a_{x(\al)}, b_{x(\al)}, c_{x(\al)}>0$, then $a<0$ and $b,c >0$, which implies that $\al$ is totally negative. The case $a_{x(\al)}, b_{x(\al)}, c_{x(\al)}<0$ is similar.

\item[2.] As $a_{x(\al)}=-a$, the claim follows immediately.

\item[3.] Follows from 2.
\end{itemize}
\end{proof}

Some aspects of the actions of $y$ and $y^2$ on elements of $M_{-n}$ and their norms are given below.

\begin{lem}\label{positive} Let $\al=\cfrac{a+\sqrt{-n}}{c}\in M_{-n}$.
\begin{enumerate}
\item[1.] If $\al$ has norm zero, then $y(\al)$ and $y^2(\al)$ are both totally positive.
\item[2.] If $\al$ is totally negative, then $y(\al)$ and $y^2(\al)$ are both totally positive with $\|\al\|<\|y(\al)\|$ and $\|\al\|< \|y^2(\al)\|$.
\item[3.] The three elements $\al$, $y(\al)$, and $y^2(\al)$ are all totally positive if and only if either ($0<a$, $a<b$, and $a<c$) or ($0>a$, $a>b$, and $a>c$).
\end{enumerate}
\end{lem}

\begin{proof}\hfill
\begin{itemize}
\item[1.] Assume that $\|\al\|=0$ (i.e. $\al=\sqrt{-n}/c$). If $c>0$, it follows from Table 1 and Remark \ref{remark 2} that $a_{y(\al)}=b>0$ and $c_{y(\al)}=b>0$ and, thus, $y(\al)$ is totally positive. Similarly, $y^2(\al)$ is totally positive. If $c<0$, a similar argument shows that $y(\al)$ and $y^2(\al)$ are both totally positive in this case as well.

\item[2.] From Table 1, notice that $a_{y(\al)}=b-a$, $b_{y(\al)}=-2a+b+c$, $c_{y(\al)}=b$, $a_{y^2(\al)}=c-a$, $b_{y^2(\al)}=c$, and $c_{y^2(\al)}=-2a+b+c$. If $a<0$ and $b,c>0$, then it is clear that all the values $a_{y(\al)}, b_{y(\al)}, c_{y(\al)}, a_{y^2(\al)}, b_{y^2(\al)}, c_{y^2(\al)}$ are positive and, therefore, both $y(\al)$ and $y^2(\al)$ are totally positive. As for the norms in this case, we have
     \begin{center}
     $\|y(\al)\|=|b-a|=b-a>-a=\|\al\|\;\; \mbox{and}\;\; \|y^2(\al)\|=|c-a|=c-a> -a=\|\al\|.$\end{center}
     The case $a>0$ and $b,c<0$ is dealt with in a similar manner.

\item[3.] Suppose that $\al$, $y(\al)$, and $y^2(\al)$ are all totally positive. Since $\al$ is totally positive, $a,b,c>0$ or $a,b,c<0$. Assume that $a,b,c>0$. Since $c_{y(\al)}=b>0$ and $y(\al)$ is totally positive, $a_{y(\al)}=b-a>0$. So $b>a$ as desired. On the other hand, since $b_{y^2(\al)}=c>0$ (and, hence, $c_{y^2(\al)}>0$) and $y^2(\al)$ is totally positive, $a_{y^2(\al)}=c-a>0$. So $c>a$ as desired. Similarly, if $a,b,c<0$, it follows that $a>b$ and $a>c$.

Conversely, suppose that $0<a$, $a<b$, and $a<c$. Since $ac>0$, $\al$ is totally positive. As $a_{y(\al)}=b-a>0$ and $c_{y(\al)}=b>0$, $y(\al)$ is totally positive too. Also, as $a_{y^2(\al)}=c-a>0$ and $b_{y^2(\al)}=c>0$ (and, hence, $c_{y^2(\al)}>0$), $y^2(\al)$ is totally positive as well. A similar argument works if $0>a$, $a>b$, and $a>c$.
\end{itemize}
\end{proof}

\begin{remark}\label{remark 3}
It is apparent from the above lemma that for any triple $\al, y(\al), y^2(\al)$ of elements of $M_{-n}$, either all three elements are totally positive, one is totally negative and the other two are totally positive, or one is of norm zero and the other two are totally positive. This remark shall show to be useful shortly. In the terminology of coset diagrams (see \cite{MZ}, \cite{Mush1}, or \cite{T} for example), the triangle whose vertices are $\al, y(\al), y^2(\al)$ always has one of three properties: either all vertices are totally positive, one vertex is totally negative and the other two are totally positive, or one vertex is of norm zero and the other two are totally positive. We chose, however, to not use the machinery of coset diagrams in this paper as things could be handled using some combinatorial arguments.
\end{remark}

\begin{lem}\label{negative}
Every orbit in $M_{-n}$ under the action of $G$ contains a totally negative element.
\end{lem}

\begin{proof}
Consider an orbit $\al^G$ for some $\al=\cfrac{a+\sqrt{-n}}{c}\in M_{-n}$. By Remark \ref{remark 2}, $\al$ is either totally negative, totally positive, or has norm zero. If $\al$ is totally negative, then there is nothing to prove. If $\al$ is totally positive, then by Lemma \ref{x}, $x(\al)\in \al^G$ is totally negative. Finally, if $\|\al\|=0$, then it follows from Lemma \ref{positive} that $y(\al)$, for instance, is totally positive and, hence from Lemma \ref{x}, $xy(\al)\in \al^G$ is totally negative.
\end{proof}

The following lemma specifies the elements of $\C$ fixed by $x$ or $y$.

\begin{lem}\label{fixed}
Upon the action of $G$ on the complex numbers $\mathbb{C}$, the only numbers fixed by $x$ are $i, i/(-1)\in M_{-1}$ and the only numbers fixed by $y$ are $(1+\sqrt{-3})/2, (-1+\sqrt{-3})/(-2)\in M_{-3}$.
\end{lem}

\begin{proof}
Let $z\in \C$ be such that $x(z)=z$. Then $z^2=-1$, which implies that $z=\pm i$. On the other hand, if $y(z)=z$, then $z^2-z+1=0$, which implies that $z=\cfrac{1\pm \sqrt{-3}}{2}$.
\end{proof}

Recall that $$T^+(-n):=\{\{\al, y(\al), y^2(\al)\}\; |\; \mbox{$\al, y(\al), y^2(\al)\in M_{-n}$ are all totally positive}\},$$
and consider the two sets of signatures of totally positive elements of $M_{-n}$ (by Lemma \ref{positive}):
$$A^+(-n):=\{(a,b,c)\in \N^3\;|\; \cfrac{a+\sqrt{-n}}{c}\in M_{-n},\; b=\cfrac{a^2+n}{c}, \; b>a, \; c>a\}$$
and
$$A^-(-n):=\{(-a,-b,-c)\in \N^3\;|\; \cfrac{a+\sqrt{-n}}{c}\in M_{-n},\; b=\cfrac{a^2+n}{c}, \; b<a, \; c<a\}.$$

We use, next, the action of the cyclic subgroup $G_y$ generated by $y$ on $M_{-n}$ induced from the action of $G$ to define an action of $G_y$ on both $A^+(-n)$ and $A^-(-n)$.

\begin{lem}\label{G_y}
Let $G_y$ be the cyclic subgroup of $G$ generated by $y$ and $A^+(-n)\neq \varnothing$. Then, $A^+(-n)$ and $A^-(-n)$ are $G_y$-sets.
\end{lem}

\begin{proof}
For an element $(a,b,c)\in A^+(-n)$, there corresponds the unique (totally positive) element $\al$ of $M_{-n}$ whose signature is $(a,b,c)$. Using this correspondence, the action of $G_y$ on $M_{-n}$ induced from the action of $G$ on $M_{-n}$ can be used to define an action of $G_y$ on $A^+(-n)$ by letting the action of $y$ takes the signature of $\al$ to the signature of $y(\al)$ (according to Table 1); that is, $y\cdot (a,b,c)=(b-a, -2a+b+c, b)$. Note that $(b-a, -2a+b+c, b)$ is an element of $A^+(-n)$ too because  $0<a_{y(\al)}=b-a$, $a_{y(\al)}=b-a<b-a+c-a=-2a+b+c=b_{y(\al)}$, and $a_{y(\al)}=b-a<b=c_{y(\al)}$. Verifying that this proposed action of $G_y$ on $A^+(-n)$ is really so is a straightforward matter. A similar proof works for $A^-(-n)$.
\end{proof}

The following two lemmas show, in particular, that the sets $A^+(-n)$ and $T^+(-n)$ are finite and give a formula that compares their respective cardinalities for $n\neq 3$.

\begin{lem}\label{a,b,c}
If $(a,b,c) \in A^+(-n)$, then $a\leq n/2$ and $b,c \leq (n+1)/2$. Furthermore, $|A^+(-n)| \leq n(n+1)/4$.
\end{lem}

\begin{proof}
Let $(a,b,c)\in A^+(-n)$. For the claimed bound on $a$, suppose to the contrary that $a>n/2$. So, $a=n/2 + t $ for some $t\geq 1/2$. Assume that $b\geq c$ (the case $b\leq c$ is treated similarly). Since $c>a$, set $c=a+s$ for some $s\in \N$. Now, $a^2+n = bc \geq c^2$ gives $(n/2+t)^2+n \geq (n/2+t+s)^2$, which implies the absurd inequality $n\geq s^2+ns +2ts\geq 1+n+2t\geq n+3$. Thus, $a\leq n/2$.

Due to the symmetry between $b$ and $c$, it suffices to prove the claimed bound for one of them, say $b$. Since $0<a<c$, $a+1\leq c $. So, $b=(a^2+n)/c \leq (a^2+n)/(a+1)$. If $a=1$, then $b\leq (n+1)/2$ and we are done in this case. Assume that $a>1$. We show first that $b< (n+2)/2$. We have the following string of implications:
\begin{align*}
a\leq n/2 &\Rightarrow 2a \leq n\\ &\Rightarrow 2a < n+2/(a-1)\\ &\Rightarrow 2a+1 < n+1 +2/(a-1)=n+(a+1)/(a-1)\\ &\Rightarrow (2a+1)(a-1)-(a+1) < n(a-1)\\ &\Rightarrow 2a^2-2a-2 < na-n\\ &\Rightarrow 2a^2+2n < na+2a +n +2\\ &\Rightarrow 2(a^2+n) < (n+2)(a+1)\\ &\Rightarrow b\leq (a^2+n)/(a+1)< (n+2)/2.
\end{align*}
Now, if $n$ is odd, then $(n+2)/2\in (1/2)+\Z$ and, so, $b\leq (n+2)/2 -1/2 =(n+1)/2$. If $n$ is even, then $(n+2)/2\in\Z$ and, so, $b \leq (n+2)/2 -1 =n/2 < (n+1)/2$. This proves the claimed upper bound of $b$ (and of $c$, by symmetry).

As for the bound on $|A^+(-n)|$, to determine any element $(a,b,c)\in A^+(-n)$ it suffices to be given $a$ and $b$ (as $c$ would then be determined by $c=(a^2+n)/b$) or to be given $a$ and $c$ (as $b$ would then be determined by $b=(a^2+n)/c$). So, the number of possible choices for $a$ and $b$ (or for $a$ and $c$) determines the possible cardinality of $A^+(-n)$. Thus, $|A^+(n)| \leq (n/2)\,((n+1)/2)=n(n+1)/4$.
\end{proof}

\begin{lem}\label{triples}\hfill

\begin{itemize}
\item[1.] $|A^+(-n)|=1$ if and only if $n=3$.
\item[2.]$|A^+(-n)| \equiv 0 \;(\mbox{mod} \;3)$ for $n\neq 3$.
\item[3.] $|T^+(-n)|=(2/3)\,|A^+(-n)|$ for $n\neq 3$
\end{itemize}
\end{lem}

\begin{proof}
\begin{itemize}
\item[1.] Let $n=3$. Since $(1,2,2)\in A^+(-3)$, $A^+(-3) \neq \varnothing$. Let $(a,b,c)\in A^+(-3)$. As $0<a\leq 3/2$ (Lemma \ref{a,b,c}), $a=1$. Since $c|(1^2+3)$ and $a<c$, $c=2$ or $4$. But $c\leq (3+1)/2$ (Lemma \ref{a,b,c}). So, $c=2$. Similarly, $b=2$. Thus, $A^+(-3)=\{(1,2,2)\}$. Conversely, assume that $A^+(-n)=\{(a,b,c)\}$. Let $\al$ be the element of $M_{-n}$ whose signature is $(a,b,c)$. By the proof of Lemma \ref{G_y}, the signature of $y(\al)$ is also in $A^+(-n)$. So, by the assumption on $A^+(-n)$, the signatures of $\al$ and $y(\al)$ are equal. This means that $\al$ is fixed by $y$. It, thus, follows from Lemma \ref{fixed} that $n=3$.

\item[2.]
Let $n\neq 3$. If $A^+(-n)=\varnothing$, then $|A^+(-n)|=0$ and we are done. Suppose that $A^+(-n)\neq \varnothing$. Let $(a,b,c)\in A^+(-n)$ and $\al$ the element of $M_{-n}$ whose signature is $(a,b,c)$. By Lemma \ref{G_y}, $G_y$ acts on $A^+(-n)$. Since the set $A^+(-n)$ is finite (by Lemma \ref{a,b,c}), the number of orbits in $A^+(-n)$ under the action of $G_y$ is finite as well. Since the totally positive triple $\{\al, y(\al), y^2(\al)\}$ in $M_{-n}$ is invariant under the action of $G_y$, so is the corresponding triple $\{(a,b,c), (b-a, -2a+b+c, b), (c-a, c, -2a+b+c)\}$ in $A^+(-n)$ under the action of $G_y$. Since $n\neq 3$, the elements of the triple $\{\al, y(\al), y^2(\al)\}$ are distinct and, thus, so are the elements of the corresponding triple $\{(a,b,c), (b-a, -2a+b+c, b), (c-a, c, -2a+b+c)\}$. This means that each orbit in $A^+(-n)$ consists precisely of three elements and, hence, $|A^+(-n)|$ is divisible by 3 as claimed.
\item[3.] Let $n\neq 3$. It is clear that the two sets $A^+(-n)$ and $A^-(-n)$ are disjoint and that there is a bijection between them. It can also be easily seen that the same arguments in parts 1 and 2 above apply also to $A^-(n)$. Let $\mathcal{O}^{G_y}(A^+(-n))$ and $\mathcal{O}^{G_y}(A^-(-n))$ be the sets of orbits in $A^+(-n)$ and $A^-(-n)$, respectively, under the action of $G_y$. It follows from the argument in the proof of Lemma \ref{G_y} and part 2 above that there is the bijection between $T^+(-n)$ and the disjoint union $\mathcal{O}^{G_y}(A^+(-n))\cup \mathcal{O}^{G_y}(A^-(-n))$ given by
$$\{\al, y(\al), y^2(\al)\} \mapsto \{(a,b,c), (b-a, -2a+b+c, b), (c-a, c, -2a+b+c)\}.$$
Since $|\mathcal{O}^{G_y}(A^+(-n))|=(1/3)\,|A^+(-n)|=(1/3)\,|A^-(-n)|=|\mathcal{O}^{G_y}(A^-(-n))|$ and the two sets of orbits are disjoint, $|T^+(-n)|=(2/3)\,|A^+(-n)|$.
\end{itemize}
\end{proof}

\begin{remark}\label{remark 4} By making use of a C$^{++}$ code that computes $A^+(-n)$ for all $1\leq n \leq 100$ with $n$ square-free, we display in Table 2 (see the Appendix) the values $d(n),|T^+(-n)|$, and $|\mathcal{O}^G(M_{-n})|$ for all such $n$.
\end{remark}

\begin{lem}\label{divisors}
For each $n\in \mathbb{N}$, the cardinality of the set $M^0_{-n}:=\{\al\in M_{-n}\,|\, \|\al\|=0\}$ is $2d(n)$.
\end{lem}

\begin{proof}
For an element $\al$ of $M_{-n}^0$, $b_\al=n/c_\al$. For $b_\al$ to be an integer, $c_\al$ must be a divisor of $n$. So,
$M_{-n}^0=\{\sqrt{-n}/c\,|\, c \; \mbox{divides}\; n\}$, which has cardinality $2d(n)$ (considering positive and negative divisors of $n$).
\end{proof}

\begin{lem}\label{norm}
For $n\neq 1$, every orbit in $M_{-n}$ must contain either a unique pair of elements of norm zero or a unique totally positive triple; while for  $n=1$, every orbit in $M_{-1}$ must contain a unique element of norm zero.
\end{lem}

\begin{proof}
We deal with the uniqueness claims at the end of the proof. In an arbitrary orbit in $M_{-n}$, let $\al_1$ be a totally negative element (by Lemma \ref{negative}). By Lemma \ref{x}, $x(\al_1)$ is totally positive. If $yx(\al_1)$ and $y^2x(\al_1)$ are both totally positive, then we have reached at the totally positive triple $(x(\al_1), yx(\al_1), y^2x(\al_1))$, and we stop. Otherwise, one (and only one, by Lemma \ref{positive}) of $yx(\al_1)$ and $y^2x(\al_1)$ is totally negative. We set such a totally negative element as $\al_2$. We claim that $\|\al_2\|<\|\al_1\|$. If $\al_2=yx(\al_1)$, then (as $y^2(\al_2)=x(\al_1)$), it follows from Lemma \ref{positive} that $$\|\al_2\|< \|y^2(\al_2)\| = \|x(\al_1)\|=\|\al_1\|.$$
If, on the other hand, $\al_2=y^2x(\al_1)$, then (as $y(\al_2)=x(\al_1)$, it follows from Lemma \ref{positive} again that
$$\|\al_2\|< \|y(\al_2)\| = \|x(\al_1)\|=\|\al_1\|.$$ Repeating this process starting at $\al_2$ this time and proceeding in this manner, we either reach a totally positive triple at some point or, else, we keep obtaining totally negative elements $\al_1, \al_2, \al_3, \dots$ in the same orbit with $$\|\al_1\|>\|\al_2\| > \|\al_3\|> \dots$$ As the sequence $\|\al_1\|, \|\al_2\| , \|\al_3\|, \dots$ is a decreasing sequence of positive integers, the sequence must terminate. That is, if we never reach a totally positive triple, then there must exist a list of elements $\al_1, \al_2, \dots, \al_m, \al_{m+1}$ in the orbit, with $m\geq 1$, such that $\al_1, \al_2, \dots, \al_m$ are totally negative and $\al_{m+1}$ has norm zero. Now, by Lemma \ref{x}, $x(\al_{m+1})$ is also of norm zero.

What we have shown so far is that in any given orbit in $M_{-n}$, there has to be either a totally positive triple or a pair of elements of norm zero. However, their is something to clarify in the case $n=1$. First, note in this case that the element $\al_{m+1}$ of norm zero must either be $i$ or $i/(-1)$ as these are the only elements of norm zero in $M_{-1}$ (Lemma \ref{divisors}). Moreover, $i$ and $i/(-1)$ are fixed by $x$ (Lemma \ref{fixed}) and, thus, $\al_{m+1}=x(\al_{m+1})$. As $i$ and $i/(-1)$ are in distinct orbits (Lemma \ref{sign}), the element $\al_{m+1}$ of norm zero we have reached at is unique in this case. Secondly, we show that no orbit in $M_{-1}$ contains a totally positive triple, i.e. $T^+(-1)$ is empty. Suppose, on the contrary that $\al=\cfrac{a+\sqrt{-1}}{c}\in T^+(-1)$ with $0<a$, $a<b$, and $a<c$ (the other case is handled similarly). As $a<b$ and $b=(a^2+1)/c$, $ac<a^2+1$. So $a(c-a)<1$, a contradiction, because $a\geq 1$ and $c-a\geq 1$. Thus, $T^+(-1)$ is empty.

As for the uniqueness of the totally positive triple in an orbit in case $n\neq 1$ (if the orbit contains one), suppose that $\{\al, y(\al), y^2(\al)\}$ is such a triple. Then, the only way we can get out of the triple is by the action of $x$, which sends each of these three elements to a totally negative element (Lemma \ref{x}). Without loss of generality, consider the totally negative element $x(\al)$. By Lemma \ref{positive}, $yx(\al)$ and $y^2x(\al)$ are both totally positive. Again the only way to get out of the triple $\{x(\al), yx(\al), y^2x(\al)\}$ is by the action of $x$. But $xx(\al)=\al$ takes us back to $\al$ and hence back to the given totally positive triple. On the other hand, $xyx(\al)$ is totally negative and, by Lemma \ref{positive}, $yxyx(\al)$ and $y^2xyx(\al)$ are both totally positive. Similarly, $xy^2x(\al)$ is totally negative and, by Lemma \ref{positive}, $yxy^2x(\al)$ and $y^2xy^2x(\al)$ are both totally positive. If we keep repeating this process, we keep reaching endlessly at triples, one of whose entries is totally negative and the other two entries are totally positive. Since the action of $G$ on the orbit is transitive, it is certain that we will never reach at any other totally positive triple other than $\{\al, y(\al), y^2(\al)\}$.
In a similar manner, we can show that if the orbit contains an element $\al$ of norm zero, then (using Lemma \ref{positive}) $\al$ and $x(\al)$ are the only elements of norm zero in the orbit.
\end{proof}

\subsection{{\bf Proofs of Theorems \ref{orbits 1} and \ref{orbits 2}}}

\begin{proof} ({\bf Theorem \ref{orbits 1}})\hfill

For $n=1$, it follows from Lemma \ref{norm} and its proof that an orbit in $M_{-1}$ must contain either $i$ or $i/(-1)$ and not both. Thus, $M_{-1}$ contains precisely two orbits. As for $n\neq 1$, Lemma \ref{norm} shows that an arbitrary orbit contains uniquely either a pair of elements of norm zero or a totally positive triple and not both. By this and Lemma \ref{divisors}, we have as claimed:
$$|\mathcal{O}^G(M_{-n})|=\frac{1}{2}|M_{-n}^0|+|T^+(-n)|=d(n) + |T^+(-n)|.$$

For $n=3$, it follows from the argument in the proof of Lemma \ref{triples} (part 1) that $$T^+(-3)=\{\cfrac{1+\sqrt{-3}}{2}, \cfrac{-1+\sqrt{-3}}{-2}\}.$$ Thus, $|\mathcal{O}^G(M_{-3})|=d(3) + |T^+(-3)|=2+2=4$.

Finally, for $n\neq 1$ or $3$, it follows from Lemma \ref{triples} (part 3) that $$|\mathcal{O}^G(M_{-n})|=d(n) + |T^+(-n)|=d(n)+(2/3)|A^+(-n)|.$$

We now prove that $|\mathcal{O}^G(M_{-n})|\equiv 0 \;(\mbox{mod}\,4)$ for $n\neq 1$ or $2$. Note that we excluded the case $n=1$ since $|\mathcal{O}^G(M_{-1})|=2$ from above, and we exclude the case $n=2$ because $T^+(-2)=\varnothing$ and, thus, $|\mathcal{O}^G(M_{-2})|=d(2)+0=2$ (if $\cfrac{a+\sqrt{-2}}{c}\in T^+(-2)$, then as $a\leq 2/2$, a=1; but then $c\leq 3/2$ and, so, $c=1$; we reject this because $a<c$; hence, $T^+(-2)=\varnothing$). Since $|\mathcal{O}^G(M_{-3})|=4$, $|\mathcal{O}^G(M_{-3})|\equiv 0 \;(\mbox{mod}\,4)$. Now, let $n>3$. By the paragraph above, we have $|\mathcal{O}^G(M_{-n})|=d(n) +(2/3)|A^+(-n)|$. It thus follows that $$|\mathcal{O}^G(M_{-n})|\equiv d(n) + 2\,|A^+(-n)|\;(\mbox{mod}\,4).$$ We write the set $A^+(-n)$ as the disjoint union of subsets in the form $$A^+(-n)=A_{b\neq c}^+(-n)\cup A_{b=c}^+(-n),$$ where
$$\mbox{$A_{b\neq c}^+(-n):=\{(a,b,c)\in A^+(-n)\;|\; b\neq c\}$ and $A_{b=c}^+(-n):=\{(a,b,c)\in A^+(-n)\;|\; b=c\}$.}$$
By Lemma \ref{a,b,c}, the two sets $A_{b\neq c}^+(-n)$, and $A_{b=c}^+(-n)$ are finite. As a general observation, we can see that $(a,b,c)\in A^+(-n)$ if and only if $(a,c,b)\in A^+(-n)$, which implies that elements in the set $A_{b\neq c}^+(-n)$ occur in pairs. Thus, $|A_{b\neq c}^+(-n)|$ is always even.

For the rest of the proof, we deal with three cases separately: when $n$ is an even composite integer, when $n$ is an odd prime, and when $n$ is an odd composite integer.
\begin{itemize}
\item[Case 1:] Let $n$ be an even composite integer with $n=2m$ for some $m>1$ with $m$ odd (as $n$ is square-free). Since $d(n)=d(2)d(m)=2d(m)$ and $2|d(m)$, $d(n)\equiv 0\; (\mbox{mod}\,4)$. So, $|\mathcal{O}^G(M_{-n})|\equiv 2|A^+(-n)|\;(\mbox{mod}\,4)$. Since $|A^+(-n)|=|A_{b\neq c}^+(-n)|+|A_{b=c}^+(-n)|$ and $|A_{b\neq c}^+(-n)|$ is even, $|\mathcal{O}^G(M_{-n})|\equiv 2|A_{b=c}^+(-n)| \;(\mbox{mod}\,4)$ in this case. Let $(a,b,b)\in A_{b=c}^+(-n)$. Then $b^2=a^2+n$, which implies that $(b+a)(b-a)=n=2m$. If $2|(b+a)$, then $b-a=m/k$, where $b+a=2k$ and $k$ is odd (as $m$ is odd). Thus, $2b=2k+m/k$ is odd, which is impossible. A similar contradiction occurs if $2|(b-a)$. We thus conclude that $A_{b=c}^+(-n) =\varnothing$ in this case and, hence, $|\mathcal{O}^G(M_{-n})|\equiv 0 \;(\mbox{mod}\,4)$.

\item[Case 2:] Let $n$ be an odd prime. So, $d(n)=2 \equiv 2 \;(\mbox{mod}\,4)$. Then, $|\mathcal{O}^G(M_{-n})|\equiv 2 + 2\,|A^+(-n)|\;(\mbox{mod}\,4)$ and, therefore, it suffices to show that $|A^+(-n)|$ is odd in this case. Since $|A^+(-n)|=|A_{b\neq c}^+(-n)|+|A_{b=c}^+(-n)|$ and $|A_{b\neq c}^+(-n)|$ is even, we show that $|A_{b=c}^+(-n)|$ is odd. We, in fact, show that $|A_{b=c}^+(-n)|=1$. For $(a,b,b)\in A_{b=c}^+(-n)$, $b^2=a^2+n$ and, thus, $(b+a)(b-a)=n$. Since $b+a>b-a$ and $n$ is prime, we must have $b+a=n$ and $b-a=1$. Thus, $b=(n+1)/2$ and $a=(n-1)/2$. That is, $((n-1)/2, (n+1)/2, (n+1)/2)$ is the only element in $A_{b=c}^+(-n)$. Hence, the claimed congruence is settled in this case too.
\item[Case 3:] Let $n$ be an odd composite integer with $n=p_1 p_2 \dots p_r$, $r\geq 2$, where the $p_i$ are distinct primes (as $n$ is square-free). Then $d(n)=d(p_1)d(p_2)\dots d(p_r)=2^r \equiv 0 \;(\mbox{mod}\,4)$.  So, $|\mathcal{O}^G(M_{-n})|\equiv 2\,|A^+(-n)|\;(\mbox{mod}\,4)$ and, therefore, it suffices to show that $|A^+(-n)|$ is even in this case. Since $|A^+(-n)|=|A_{b\neq c}^+(-n)|+|A_{b=c}^+(-n)|$ and $|A_{b\neq c}^+(-n)|$ is even,, we show that $|A_{b=c}^+(-n)|$ is even as well. In fact, we prove the following stronger claim:
    $$|A_{b=c}^+(-n)|= \left \{ \begin{array} {l@{\quad;\quad}l} C^0_r + C^1_r +\dots + C^{\frac{r}{2}-1}_r+\frac{1}{2}C^{\frac{r}{2}}_r & \mbox{if $r$ is even} \\
C^0_r + C^1_r + \dots + C^{\frac{r-1}{2}-1}_r + C^{\frac{r-1}{2}}_r & \mbox{if $r$ is odd}.
\end{array} \right. $$
    For $(a,b,b)\in A_{b=c}^+(-n)$, $b^2=a^2+n$ and, thus, $(b+a)(b-a)=n=p_1p_2\dots p_r$. We notice that $b+a>b-a$ and investigate all the possible ways of factoring $b+a$ and $b-a$. Suppose that $r$ is even. Then, there is $C^0_r$ possibility that $b+a$ is the product of $r$ primes (i.e. $a+b=n$) and $b-a$ is the product of no primes (i.e. $b-a=1$), and there is $C^1_r$ possibilities that $b+a$ is the product of $r-1$ primes and $b-a$ is the product of one prime. We continue in this manner until we get to the final scenario which is having $\frac{1}{2}C^{\frac{r}{2}}_r$ possibilities of writing both of $b+a$ and $b-a$ as a product of $r/2$ primes each. Seeing obviously that each single possibility among the above ways of factorizations of $b+a$ and $b-a$ corresponds uniquely to a single point of $A_{b=c}^+(-n)$, the conclusion of the claim when $r$ is even follows immediately. The case when $r$ is odd is handled similarly. From elementary combinatorics (see \cite{Ros} for instance), we know that $\sum_{k=0}^r C^k_r=2^r$ and $C^k_r=C^{r-k}_r$ for $k=0, \dots, r$. So, if $r$ is even, then $C^0_r + C^1_r +\dots + C^{\frac{r}{2}-1}_r+\frac{1}{2}C^{\frac{r}{2}}_r=\frac{1}{2}C^{\frac{r}{2}}_r+C^{\frac{r}{2}+1}_r +\dots + C^r_r$. Thus, $2^r=\sum_{k=0}^r C^k_r = 2\left(C^0_r + C^1_r +\dots + C^{\frac{r}{2}-1}_r+\frac{1}{2}C^{\frac{r}{2}}_r\right)=2|A_{b=c}^+(-n)|$. Hence, $|A_{b=c}^+(-n)|=2^{r-1}$ which is even as desired. The same conclusion is reached similarly if $r$ is odd. This concludes the proof.
\end{itemize}
\end{proof}
\newpage
\begin{proof} ({\bf Theorem \ref{orbits 2}})\hfill

By Theorem \ref{orbits 1}, $|\mathcal{O}^G(M_{-n})|=d(n)+2/3 |A^+(-n)|$. So the desired claim of the current theorem holds if and only if $$|A^+(-n)|=\sum_{i=1}^{\lfloor n/2 \rfloor}[ d(i^2+n)-2d_{_{\leq i}} (i^2+n)].$$ We seek now to prove this last equality. Making use of Lemma \ref{a,b,c}, we first write the set $A^+(-n)$ as a disjoint union of subsets in the form $$A^+(-n)=A_1^+(-n)\cup A_2^+(-n) \cup \dots \cup A_{\lfloor n/2 \rfloor}^+(-n),$$ where, for each $i=1,2, \dots, \lfloor n/2 \rfloor$, $$A_i^+(-n):=\{(i,b,c)\in \N^3\;|\; i<b, i<c, \;\mbox{and}\; b=(i^2+n)/c\}.$$ For a fixed such $i$, we can see that $A_i^+(-n)=A_{i,d_1}^+(-n)-\left\{A_{i, d_2}^+(-n)\cup A_{i,d_3}^+(-n)\right\}$, where
\begin{align*}
A_{i,d_1}^+(-n)&:= \{(i,d_1, (i^2+n)/d_1)\in A_i^+(-n)\;|\; d_1\in \N \; \mbox{and}\; d_1 |(i^2+n)\},\\
A_{i,d_2}^+(-n)&:= \{ (i,d_2, (i^2+n)/d_2)\in A_i^+(-n)\;|\; d_2\in \N, d_2\leq i, \;\mbox{and}\; d_2|(i^2+n)\},\\
A_{i,d_3}^+(-n)&:= \{(i, (i^2+n)/d_3, d_3) \in A_i^+(-n) \;|\: d_3\in \N, d_3\leq i, \; \mbox{and}\; d_3|(i^2+n)\}.
\end{align*}
Note that $|A_{i,d_1}^+(-n)|=d(i^2+n)$ and $|A_{i,d_2}^+(-n)|=|A_{i,d_3}^+(-n)|=d_{_{\leq i}}(i^2+n)$. If the latter two sets have a point in common, then for some $d_2\leq i$ and $d_3\leq i$ we would have $d_2d_3=i^2+n \leq i^2$, which is absurd. So, these two sets are disjoint and, hence, $$|A_i^+(-n)|=|A_{i,d_1}^+(-n)|-|A_{i,d_2}^+(-n)|-|A_{i,d_3}^+(-n)|=d(i^2+n)-2d_{_{\leq i}}(i^2+n).$$
As $|A^+(-n)|=|A_1^+(-n)|+ |A_2^+(-n)| + \dots + |A_{\lfloor n/2 \rfloor}^+(-n)|$, the desired equality follows.
\end{proof}

\begin{cor}
The action of $G$ on $M_{-n}$ is intransitive for any square-free $n\in \N$.
\end{cor}

\begin{example}\label{example 2}
As an illustration, we compute in this example the value $|\mathcal{O}^G(M_{-n})|$ for $n=11$ in such a way that verifies both Theorem \ref{orbits 1} and Theorem \ref{orbits 2} in this case.

By Theorem \ref{orbits 1} and its proof, $|\mathcal{O}^G(M_{-11})|=d(11)+|T^+(-11)|=d(11)+(2/3)\,|A^+(-11)|$.
Of course, $d(11)=2$. So, it remains to find $|A^+(-11)|$. By Lemma \ref{a,b,c}, for $(a,b,c)\in A^+(-11)$, $a\leq 5$ and $c\leq 6$. We try these values one by one. For $a=1$, $(1^2+11)/c\in \N$ if and only if $c|12$. So, by Lemma \ref{a,b,c} again, the possible candidate values of $c$ are $1, 2, 3, 4$, and $6$. Since $a<c$, we discard the value $c=1$. For $c=2$, we have $b=6$ and we get that $(1,2,6)\in A^+(-11)$. For $c=3$, we have $b=4$ and we get that $(1, 3, 4)\in A^+(-11)$. For $c=4$, we have $b=3$ and we get that $(1, 4,3)\in A^+(-11)$. For $c=6$, we have $b=2$ and we get that $(1, 6, 2)\in A^+(-11)$. For $a=2$, $(2^2+11)/c\in \N$ if and only if $c|15$. By an argument similar to the above, we get in this case only two elements $(2, 3,5), (2,5,3)\in A^+(-11)$. For $a=3$, $(3^2+11)/c\in \N$ if and only if $c|20$. We also get in this case only two elements $(3,4,5), (3,5,4)\in A^+(-11)$. For $a=4$, $(4^2+11)/c\in \N$ if and only if $c|27$. The values $c=1$ and $3$ are discarded as $a<c$. Thus, for $a=4$ we get no element in $A^+(-11)$. For $a=5$, it can be checked similarly that we only get only the element $(5,6,6)\in A^+(-11)$. In summary, we have $|A^+(-11)|=9$ and, thus, $|\mathcal{O}^G(M_{-11})|=d(11)+(2/3)(9)=8$.

On the other hand, by Theorem \ref{orbits 2}, we have
\begin{align*}
|\mathcal{O}^G(M_{-11})|&=d(11)+ \cfrac{2}{3}\,\sum_{i=1}^5[ d(i^2+11)-2d_{\leq i} (i^2+11)]\\
& =2+(2/3)\, \{\,[d(12)+d(15)+d(20)+d(27)+d(36)]\\
& \quad -2\,[d_{\leq 1}(12)+d_{\leq 2}(15)+d_{\leq 3}(20)+d_{\leq 4}(27)+d_{\leq 5}(36)]\,\}\\
& =2+(2/3)\left\{[6+4+6+4+9]-2\,[1+1+2+2+4]\right\}\\
& =2+(2/3)(29-20)\\
& =8.
\end{align*}
\end{example}

\section*{Acknowledgement}
The authors would like to express their gratitude to King Khalid University for providing administrative and technical support. The second author would also like to thank the University Council and the Scientific Council of King Khalid University for approving a sabbatical leave request for the academic year 2018-2019, during which this article was prepared and submitted.

\section*{Appendix}
Using a C$^{++}$ code to compute the sets $A^+(-n)$ for all $1\leq n \leq 100$ with $n$ square-free, the following table gives the values of $|T^+(-n)|, d(n)$, and $|\mathcal{O}^G(M_{-n})|$ for all such $n$.

\begin{center}
\begin{tabular}{|c| c c c | c| c c c|c| c c c| }
\hline
 n & $|T^+(-n)|$ & $d(n)$ & $|\mathcal{O}^G(M_{-n})|$ &  n & $|T^+(-n)|$ & $d(n)$ & $|\mathcal{O}^G(M_{-n})|$ &  n & $|T^+(-n)|$ & $d(n)$ & $|\mathcal{O}^G(M_{-n})|$\\
\hline
1 & 0 & 2 & 2 & 33 & 4 & 4 & 8 & 67 & 6 & 2 & 8 \\
2 & 0 & 2 & 2 & 34 & 4 & 4 & 8 & 69 & 12 & 4 & 16\\
3 & 2 & 2 & 4 & 35 & 12 & 4 & 16 & 70 & 0 & 8 & 8\\
5 & 2 & 2 & 4 & 37 & 2 & 2 & 4 & 71 & 26 & 2 & 28\\
6 & 0 & 4 & 4 & 38 & 8 & 4 & 12 & 73 & 6 & 2 & 8\\
7 & 2 & 2 & 4 & 39 & 12 & 4 & 16 & 74 & 16 & 4 & 20\\
10 & 0 & 4 & 4 & 41 & 14 & 2 & 16 & 77 & 12 & 4 & 16   \\
11 & 6 & 2 & 8 & 42 & 0 & 8 & 8 & 78 & 0 & 8 & 8 \\
13 & 2 & 2 & 4 & 43 & 6 & 2 & 8 & 79 & 18 & 2 & 20 \\
14 & 4 & 4 & 8 & 46 & 4 & 4 & 8 & 82 & 4 & 4 & 8  \\
15 & 4 & 4 & 8 & 47 & 18 & 2 & 20 & 83 & 22 & 2 & 24 \\
17 & 6 & 2 & 8 & 51 & 12 & 4 & 16 & 85 & 4 & 4 & 8 \\
19 & 6 & 2 & 8 & 53 & 10 & 2 & 12 & 86 & 16 & 4 & 20 \\
21 & 4 & 4 & 8 & 55 & 12 & 4 & 16 & 87 & 20 & 4 & 24 \\
22 & 0 & 4 & 4 & 57 & 4 & 4 & 8 & 89 & 22 & 2 & 24 \\
23 & 10 & 2 & 12 & 59 & 22 & 2 & 24 & 91 & 12 & 4 & 16 \\
26 & 8 & 4 & 12 & 61 & 10 & 2 & 12 & 93 & 4 & 4 & 8 \\
29 & 10 & 2 & 12 & 62 & 12 & 4 & 16 & 94 & 12 & 4 & 16 \\
30 & 0 & 8 & 8 & 65 & 12 & 4 & 16 & 95 & 28 & 4 & 32 \\
31 & 10 & 2 & 12 & 66 & 8 & 8 & 16 & 97 & 6 & 2 & 8 \\

\hline
\end{tabular}
\end{center}
\begin{center} \tablename{ 2}: The number of orbits in $M_{-n}$ for square-free $1\leq n \leq 100$ \end{center}$\\$

Below is the C$^{++}$ code used to compute the sets $A^+(-n)$ for $1\leq n\leq 100$. \\

\fontsize{10}{10}\selectfont
{\tt
$\#$include<iostream> using namespace std;

\indent\indent int main $()\{$

\indent\indent\indent int $n, a, b,c,$count $=0$,check $=0$;

\indent\indent\indent for $(n=1; n<101; n^{++})\{$

\indent\indent\indent\indent if $((n\% 4!=0)\&\&(n\%9!=0)\&\&(n\%25!=0)\&\&(n\%49!=0))\{$

\indent\indent\indent\indent\indent for $(a=1; a<100; a^{++})\{$

\indent\indent\indent\indent\indent\indent for $(b=2; b<100; b^{++})\{$

\indent\indent\indent\indent\indent\indent\indent for $(c=2; c<100; c^{++})\{$

\indent\indent\indent\indent\indent\indent\indent\indent if $((b>a) \&\& (c>a))\{$

\indent\indent\indent\indent\indent\indent\indent\indent\indent if $((b^*c-a^*a)==n)\{$

\indent\indent\indent\indent\indent\indent\indent\indent\indent\indent cout$<<$"when $n=$"$<<n<<$"$,a=$"$<<a<<$"$,b=$"

\indent\indent\indent\indent\indent\indent\indent\indent\indent\indent\indent $<<b<<$"$,c=$"$<<c<<$endl;

\indent\indent\indent\indent\indent\indent\indent\indent\indent\indent count++;

\indent\indent\indent\indent\indent\indent\indent\indent\indent\indent check$=1$;

\indent\indent\indent\indent\indent\indent\indent\indent\indent $\}$

\indent\indent\indent\indent\indent\indent\indent\indent $\}$

\indent\indent\indent\indent\indent\indent\indent $\}$

\indent\indent\indent\indent\indent\indent $\}$

\indent\indent\indent\indent\indent $\}$

\indent\indent\indent\indent\indent if (check$==1$)$\{$

\indent\indent\indent\indent\indent\indent cout$<<$"Possibilities for"$<<n<<$":"$<<$count$<<$endl$<<$endl;

\indent\indent\indent\indent\indent\indent count $=0$;

\indent\indent\indent\indent\indent\indent check $=0$;

\indent\indent\indent\indent\indent $\}$

\indent\indent\indent\indent $\}$

\indent\indent\indent $\}$

\indent\indent\indent return $0$;

\indent\indent $\}$

}

\end{document}